%% file: arwcritical.tex
\documentclass[a4paper,12pt,oneside]{article}

\input{macros.tex}

\begin{document}

\title{Non-equilibrium Phase Transitions: \\ Activated Random Walks at Criticality}
\author{M. Cabezas, L. T. Rolla, V. Sidoravicius \\ \small Instituto de Matem\'atica Pura e Aplicada and Universidad de Buenos Aires}
\date{}
\maketitle

\begin{abstract}
In this paper we present rigorous results on the critical behavior of the Activated Random Walk model.
We conjecture that on a general class of graphs, including~$\mathbb Z^d$, and under general initial conditions, the system at the critical point does not reach an absorbing state.
We prove this for the case where the sleep rate~$\lambda$ is infinite.
Moreover, for the one-dimensional asymmetric system, we identify the scaling limit of the flow through the origin at criticality.
The case $\lambda < + \infty$ remains largely open, with the exception of the one-dimensional totally-asymmetric case, for which it is known that there is no fixation at criticality.
\end{abstract}

This preprint has the same numbering of sections, equations and theorems as the the
published article
``\emph{J. Stat. Phys. 155 (2014), 1112-1125.}''

\section{Introduction}
\label{sec:intro}

In this work we investigate the behavior of the Activated Random Walk (ARW) model at the critical density of particles.
Along with the fixed-energy sandpiles with stochastic update rules, the ARW constitutes one of the paradigm  examples of conservative lattice gases which exhibit non-equilibrium phase transition from an active phase into infinitely many absorbing states. It is believed that the transitions in these models belong to an autonomous universality class of non-equilibrium phase transitions, the so-called Manna class. While the existence of such transition is broadly supported numerically, rigorously it is proven only for few particular cases.
Much less is known about the behavior of these systems at the critical point.
For comprehensive background and historical remarks we refer to~\cite{marro-dickman-99,dickman-rolla-sidoravicius-10,rolla-sidoravicius-12}, and for ongoing discussion on  the existence of an independent Manna class, see~\cite{basu-basu-bondyopadhyay-mohanty-hinrichsen-12,lee-13}.

Our goal here is to show that, at the critical density,
the system does not reach an absorbing state,
and thus each individual particle is in active state for infinitely many time intervals (although the density of active particles vanishes as $t\to\infty$).
We prove this for some particular cases, as discussed below.
We also compute the critical exponent and find the scaling limit for a one-dimensional model.

The ARW is defined as follows.
Initially, there are infinitely many particles spread over~$\Z^d$ with density $\mu$, e.g.\ i.i.d.\ Poisson with mean~$\mu$.
Particles can be in state $A$ for \emph{active} or $S$ for \emph{passive}, and at $t=0^-$ they are all active.
Each active particle, that is, each particle in the $A$ state, performs a continuous-time random walk with jump rate $D_A=1$ and with translation-invariant jump distribution.
Several active particles can be at the same site, and they do not interact among themselves.
When a particle is alone, it may become passive, a transition denoted by $A\to S$, which occurs at a \emph{sleeping rate} $\lambda>0$.
In other words, each particle carries two clocks, one for jumping and one for sleeping.
Once a particle is passive, it stops moving, i.e., it has jump rate $D_S=0$, and it remains passive until the instant when another particle is present at the same vertex.
At such an instant the particle which is in $S$ state flips to the $A$ state, giving the transition $A+S \to 2A$.
A particle in the $S$ state stands still forever if no other particle ever visits the vertex where it is located.
At the extreme case $\lambda=+\infty$, when a particle visits an empty site, it becomes passive instantaneously.
This case is thus essentially equivalent to \emph{internal diffusion-limited aggregation} with infinitely many sources.
For a more formal definition of the model, see~\cite[Section~2]{rolla-sidoravicius-12}.

In this paper we are mostly concerned with the question of fixation.
We say that the system \emph{fixates} if, for every finite box, there exists a random time after which there is no activity in that box.

The behavior of the ARW is expected to be the following.
For each~$0<\lambda<\infty$ there exists~$0<\mu_c<1$ such that, if the initial density~$\mu$ of particles satisfies $\mu<\mu_c$ the system fixates, and if $\mu>\mu_c$ the system does not fixate.
The critical density satisfies $\mu_c \to 0$ as $\lambda \to 0$ and $\mu_c \to 1$ as $\lambda \to \infty$.
The value of $\mu_c$ should not depend on the particular \mbox{$\mu$-}parametrized distribution of the initial configuration (geometric, Poisson, etc.).
At $\mu=\mu_c$, the density of active particles vanishes as $t\to\infty$, but we conjecture that the system does not fixate in this case.
The asymptotic decay of density of activity as $t \gg 1$ when $\mu=\mu_c$ should obey a power law.
Also, for the stationary regime, i.e., letting $t\to\infty$ first, the density of activity should decay with a power law as $0 < \mu-\mu_c \ll 1$.

However, from a mathematically rigorous point of view, all the above predictions are still open problems.
The few exceptions are presented in the next sections.

The only existing general result is for $d=1$, and states that $0 < \mu_c \leq 1$, but there is no proof that $\mu_c<1$.
Here we consider two cases: infinite sleep rate $\lambda=\infty$, and a one-dimensional system with totally asymmetric jumps.
These extreme cases share some but not all of the qualitative aspects of the ARW.

In Section~\ref{s:models} we define the particle-hole model and discuss its relation with the ARW.
We present two alternative constructions for these models, and describe the Abelian property.

In Section~\ref{sec:phasetransition} we present known and new results concerning fixation and non-fixation for the ARW.
For a quick glance the reader may look at the main statements.
We also comment on open problems.

In Section~\ref{sec:criticalflow} we study the scaling limit for the flow of particles in the asymmetric
one-dimensional particle-hole model, and finish with a brief discussion of this scaling limit.

The theorems presented here are valid for any dimension, any nearest-neighbor jump distribution, and any value of~$\lambda$, unless different assumptions are explicitly stated.

\section{Models, time-change, and fixation}
\label{s:models}

We start introducing two related models which will be helpful in the study of the ARW, and then discuss the Abelian property.

All the systems considered in this paper, including the ARW, are particle systems on~$\Z^d$.
The configuration~$(\eta_t(x) : x\in{\Z^d})$ denotes the state of the system at each site~$x$ at time~$t$.
The particle jumps are distributed as $p(x,x+y)=p(\o,y)=p(y)$, where $p(\cdot)$ is a probability function on $\Z^d$ and~$\o$ denotes the origin in~$\Z^d$.
It is assumed that the initial configuration~$\eta_0$ is i.i.d.~with finite mean~$\mu$ and non-constant.
For simplicity we also assume that~$p(y)=0$ unless~$y$ is one of the~$2d$ nearest-neighbors of~$\o$.

\emph{Particle-hole model.}
Each particle performs a continuous-time random walk at rate~$1$.
We refer to the sites not containing particles as \emph{holes}.
At the time when the particle visits an empty site, it \emph{settles}, i.e., it stops moving and stands still at that site forever after.
After that time the site becomes available for other particles to go through.
If a site is occupied by several particles at $t=0^-$, we choose one of them uniformly to fill the hole at $t=0$, and the other particles remain free to move.
In this setting, particles can be either \emph{unsettled} if they have never stepped on an unoccupied site, or otherwise they are \emph{settled} at some site if they have filled the corresponding hole.

\emph{Two-type annihilating random walks.}
There are two types of particles, $A$ and $B$. The particles evolve according to continuous-time random walks at rates $D_A=1$ and $D_B \geq 0$, respectively.
When two particles of different types meet, both are removed from the system.
If a particle meets several particles of the other type, it chooses one of them uniformly to annihilate.

For special choices of parameters, the above models and the ARW are closely related.

The connection between the particle-hole model and the two-type annihilating random walks is more evident.
For the latter, suppose that at $t=0^-$ every site contains one $B$-particle and that $D_B=0$, i.e., $B$-particles do not move.
If we identify $A$-particles at $t=0^-$ with unsettled particles, and $B$-particles at $t=0^-$ with holes, the evolutions of both systems will be identical.
At $t=0$, sites containing $A$-particles loose one $A$-particle, which is annihilated by the only $B$-particle present at $t=0^-$ (resp.\ one of the unsettled particles settles and fills the corresponding hole).
At positive times, each site~$x$ containing~$k$ particles of type~$A$ (resp.\ $k$ unsettled particles) sends a particle to $z=x+y$ at rate $k\cdot p(y)$.
If the target site $z$ still contains a $B$-particle (resp.\ has an unfilled hole), both particles annihilate each other (resp.\ the unsettled particle settles at $z$).

Now consider the ARW with $\lambda=\infty$.
Compared to the particle-hole model, sites containing~$0$ particles are equivalent to a hole, and sites containing~$1$ particle are equivalent to a site with one settled particle.
In both cases, particles are not sent to neighboring sites at any rate.
Sites~$x$ with~$k+1$ particles correspond to sites with one settled particle and $k \geq 1$ unsettled particles.
In this case, a particle is sent to a neighboring site at rate $k+1$ for the ARW and at rate~$k$ for the particle-hole model, and the target site is chosen as $z=x+y$ with probability $p(y)$.

The continuous-time evolution of the ARW and the particle-hole model are thus different.
However, due to the Abelian property of those models, fixation for the ARW with $\lambda=\infty$ is equivalent to fixation for the particle-hole model.

We now describe the Abelian property.
The reader can find the details in~\cite[Section~3]{rolla-sidoravicius-12}.
These models can be constructed explicitly in a number of ways.\footnote
{The constructions described here are well-defined even when the total number of particles in the system is infinite, see~\cite{cabezas-rolla-sidoravicius-13b,rolla-sidoravicius-12}.
Alternatively, one can approximate the probability of any event by a construction with only finitely many particles.}
One way, which we refer to as the \emph{particle-wise randomness}, is as follows.
First sample the number of particles present at each site at $t=0^-$, and then sample, for each particle, a continuous-time trajectory (with an extra rate-$\lambda$ Poisson clock to make the particle sleep in case $\lambda<\infty$).
A particle follows the corresponding trajectory until it settles or goes to sleep, and at such moment we stop progressing in both its trajectory and sleep clock.

Another construction, which we refer to as the \emph{site-wise randomness}, is as follows.
First sample the number of particles present at each site at $t=0^-$, and then sample, for each site~$x$, a sequence of instructions and a Poisson clock.
We now progress in the Poisson clock of each site~$x$ with a speed proportional to the number of active or unsettled particles present at~$x$.
Each time a new Poissonian mark if found in the clock of a given site~$x$, we \emph{topple site $x$}, that is, we perform the action indicated by the first unused instruction in the sequence assigned to~$x$.

The particle-wise randomness is especially useful for playing with particle addition and deletion in the particle-hole model.
The site-wise randomness, on the other hand, has the big advantage of decoupling the property of fixation from the order at which topplings are performed, and is thus robust with respect to the details of jump rates, etc.
In particular, one can combine the site-wise randomness to get equivalence between both models and the particle-wise randomness to prove fixation or non-fixation for one of them.

From now on we discuss combinatorial properties of the toppling operation, and conclude this section by stating the relationship between such properties and fixation for the continuous-time models.

A site~$x$ is \emph{unstable} in a given configuration~$\eta$ if~$x$ contains active or unsettled particles in that configuration, and \emph{stable} otherwise.
Let $\alpha$ denote finite sequences of sites in~$\Z^d$, which we think of as the order at which a sequence of topplings will be applied.
Toppling a site is \emph{legal} if the site is unstable, and a sequence $\alpha$ is said to be legal if each subsequent toppling is legal.
Let~$V$ denote finite subsets of~$\Z^d$. A configuration $\eta$ is said to be \emph{stable} in~$V$ if all the sites $x\in V$ are stable in $\eta$.
We say that \emph{$\alpha$ is contained in~$V$} if all its elements are in~$V$.
We say that \emph{$\alpha$ stabilizes $\eta$ in $V$} if every $x\in V$ is stable after performing the topplings indicated in~$\alpha$.
Let $m_\alpha(x) $ count the number of times that a site $x\in\Z^d$ appears in $\alpha$.

The main property of this construction is that the order at which topplings are performed is irrelevant.
In order to stabilize a configuration~$\eta$ in a box~$V$, the number of topplings performed at each site depends only on the sequences of instructions.

\begin{lemma}[Abelian Property]
\label{lemma2abelianproperty}
If~$\alpha$ and~$\beta$ are both legal toppling sequences for $\eta$ that are contained in~$V$ and stabilize~$\eta$ in~$V$, then $m_\alpha(x)=m_\beta(x)\ \forall x \in \Z^d$.
\end{lemma}

We can therefore define the random fields $m_{V} = m_\alpha$, which do not depend on the particular choice of~$\alpha$ that is legal and stabilizing for~$\eta$ in~$V$.
These fields depend on the randomness only through~$\eta$ and the sequences of instructions.

In particular, since a configuration~$\eta$ is stable in the ARW with $\lambda=\infty$ if and only if it is stable in the particle hole-model, $m_{V}$ is the same for both models.

\begin{lemma}[Monotonicity]
\label{lemma3generalmonotonicity}
If $V\subseteq V'$, then $m_{V}(x)\leqslant m_{V'}(x)\ \forall x\in\Z^d$.
\end{lemma}

In particular, the limit $m = \lim_{n}m_{V_n}$ exists and does not depend on the particular sequence $V_n\uparrow\Z^d$.

\begin{lemma}
\label{lemma4fixationstabilizable}
For both the ARW and the particle-hole model,
with i.i.d.\ initial configuration,
\[
\Pb[\o \mbox{ is visited finitely often}]
=
\Pb[m(\o)<\infty]
=
0 \mbox{ or } 1
.
\]
In particular, fixation for the ARW with $\lambda=\infty$ is equivalent to fixation for the particle-hole model, since $m$ is the same for both models.
\end{lemma}

\section{Critical behavior of activated random walks}
\label{sec:phasetransition}

We start with an ``exactly solvable'' case, for which a more complete description can be derived.
The following result comes from discussions with C.~Hoffman.

\begin{theorem}
\label{thm:mu_c}
For the one-dimensional totally-asymmetric ARW, $\mu_c = \frac{\lambda}{1+\lambda}$.
Moreover, this system does not fixate at criticality.
\end{theorem}
\begin{proof}
We know from Lemma~\ref{lemma4fixationstabilizable} that fixation is equivalent to
\[
\Pb[m(\boldsymbol o)\geq k] \to 0
\quad
\mbox{ as }
\quad
k\to\infty
,
\]
or alternatively
\[
\Pb[\mbox{more than $k$ particles ever jump out of $\o$}] \to 0
\quad
\mbox{ as }
\quad
k\to\infty
.
\]
Fix some $L\in\N$.
Let the site $x=-L$ topple until it is stable, and denote by $Y_0$ the indicator of the event that the last particle remained passive on $x=-L$.
Conditioned on $\eta_0(-L)$, the distribution of $Y_0$ is Bernoulli with parameter $\frac{\lambda}{1+\lambda}$ (in case $\eta_0(-L)=0$, sample $Y_0$ independently of anything else). Define $N_0$ be the number of particles which jump from $x=-L$  to $x=-L+1$, that is $N_0:=[\eta_0(-L)-Y_0]^+$.
Note that, after stabilizing $x=-L$, there are $N_0+\eta_0(-L+1)$ particles at $x=-L+1$.
Let the site $x=-L+1$ topple until it is stable, and denote by $Y_1$ the indicator of the event that the last particle remained passive on $x=-L+1$.
Again, conditioned on $\eta_0(-L)$, $\eta_0(-L+1)$, and $Y_0$, the distribution of $Y_1$ is Bernoulli with parameter $\frac{\lambda}{1+\lambda}$ (in case $N_0+\eta_0(-L+1)=0$, sample $Y_1$ independently of anything else).
Let $N_1$ be the the number of particles which jump from $x=-L+1$ to $x=-L+2$, i.e., $N_1:=[N_0+\eta_0(-L+1)-Y_1]^+$.
By iterating this procedure, the number $N_i$ of particles which jump from $x=-L+i$ to $x=-L+i+1$ after stabilizing $x=-L,\dots,-L-i$ is given by $N_i=[N_{i-1}+\eta_0(-L+i)-Y_i]^+$.
Note that the process $(N_i)_{i=0,1,\dots,L}$ is a random walk with independent jumps distributed as $\eta(x)-Y$, reflected at~$0$.

Now observe that $\E[\eta(-L+k)-Y_k]=\mu - \frac{\lambda}{1+\lambda}$.
If this quantity is positive, the reflected random walk is transient, and~$\Pb[N_{L} \geq \frac{1}{2}(\mu - \frac{\lambda}{1+\lambda})L]\to 1$ as $L\to\infty$, which, by Lemma~\ref{lemma4fixationstabilizable}, implies non-fixation.
On the other hand, if $\mu - \frac{\lambda}{1+\lambda}<0$, the reflected random walk is positive recurrent, and as $L\to\infty$, $N_{L}$ converges in distribution to a finite random variable, which implies tightness of $N_{L}$.
Therefore, by Lemma~\ref{lemma4fixationstabilizable},  we have fixation.

Finally, at criticality $\E[\eta(-L+k)-Y_k]=0$.
Then the reflected random walk $(N_i)_{i\geq0}$ defined above is null-recurrent and, as $L\to\infty$, $N_{L}$ converges in probability to $+\infty$, which implies non-fixation.
\end{proof}

The above theorem provides good support for the predictions discussed in Section~\ref{sec:intro}.
We now turn our attention to more general results about fixation.

\begin{theorem}
[\cite{rolla-sidoravicius-12}]
\label{thm:rsphasetransition}
For $d=1$ and $\mu < \frac{\lambda}{1+\lambda}$, the ARW fixates.
\end{theorem}

Theorems~\ref{thm:mu_c} and~\ref{thm:rsphasetransition} are the only available results for finite~$\lambda$.
The problem of fixation for some $\mu>0$ and some $\lambda<\infty$ is still wide open in higher dimensions.
In the sequel we consider~$\lambda=\infty$.

\begin{theorem}
[\cite{shellef-10}]
For $\lambda=\infty$ and $\mu$ small enough, the ARW fixates.
\end{theorem}

Using a mass-conservation argument we push this result to a sharp estimate.

\begin{theorem}
\label{thm:fixationbelowone}
For $\lambda=\infty$ and $\mu<1$, the ARW fixates.
\end{theorem}
\begin{proof}
The proof makes use of the spatially ergodic, continuous-time evolution of the particle-hole model.
We follow the idea introduced in~\cite{cabezas-rolla-sidoravicius-13b}.

We claim that for $\mu<1$ some holes are never filled, as a consequence of the mass-transport principle.
Indeed, the density of holes that are filled by time~$t$ equals the density of particles settled by time~$t$, and thus for any $t\geq 0$
\[ \Pb[\boldsymbol{o} \mbox{ contains an unfilled hole at time } t] \geq 1-\mu>0. \]
To see that, let $A(x,y)$ denote the event that a particle starting at~$x$ settles at~$y$ by time~$t$, and let $w(x,y)=\I_{A(x,y)}$.
Then $\sum_y w(x,y)$ equals the number of particles starting at $x$ that have settled by time $t$, and $\sum_y w(y,x)$ is the indicator that the hole at $x$ is filled by time~$t$.
Translation-invariance implies that $\E[\sum_y w(x,y)]=\E[\sum_y w(y,x)]$, and therefore the probability of the latter event is bounded from above by the density of particles at $t=0^-$, yielding the above inequality.
Finally, letting~$t\to\infty$ we see that some holes are never filled, proving the claim.

Therefore, the probability that $\o$ is visited finitely many times in the particle-hole model is positive, and finally by Lemma~\ref{lemma4fixationstabilizable} the ARW with $\lambda=\infty$ fixates.
\end{proof}

From now on we consider results on non-fixation.
All the known approaches work for $\lambda=\infty$ and, by monotonicity, imply non-fixation for any~$\lambda$.
With the exception of Theorem~\ref{thm:mu_c}, proving non-fixation for some $\lambda>0$ and some $\mu<1$ is still an open problem, in any dimension.

\begin{theorem}
[\cite{shellef-10,amir-gurelgurevich-10}]
\label{thm:supercriticalnon-fixation}
For $\mu>1$ the ARW does not fixate.
\end{theorem}

Comparing Theorems~\ref{thm:fixationbelowone} and~\ref{thm:supercriticalnon-fixation}, we get
\[ \mu_c = 1 \quad \mbox{for} \quad \lambda=\infty. \]
The result below implies non-fixation at criticality for this case.
\begin{theorem}
\label{thm:criticalnon-fixation}
For $\mu=1$ the ARW does not fixate.
\end{theorem}
\begin{proof}
By monotonicity in~$\lambda$ it suffices to consider $\lambda=\infty$.
By Lemma~\ref{lemma4fixationstabilizable}, the theorem follows from Proposition~\ref{prop:abrecurrence} below.
\end{proof}

\begin{proposition}
\label{prop:abrecurrence}
If the particle-hole model fixates, then necessarily~$\mu<1$.
\end{proposition}

In the sequel we give the proof of Proposition~\ref{prop:abrecurrence} following the lines of~\cite{cabezas-rolla-sidoravicius-13b}, where the equivalent model of two-type annihilating random walks is considered.
The proof uses a surgery technique.

\begin{lemma}
\label{lem:originnevervisited}
If the particle-hole model fixates, then $\Pb\left[ \o \mbox{ is never visited} \right] > 0$.
\end{lemma}
\begin{proof}
Consider the \emph{particle-wise} construction described in Section~\ref{s:models}.
We denote by $(X^{x,i}_t)_{t\geq0}$ the \emph{trajectory} assigned to the $i$-th particle potentially present at $x$ at $t=0^-$.
We will refer to the set of trajectories as the \emph{evolution rules}.
The trajectories are independent over~$x$ and~$i$ and independent of the \emph{initial configuration}~$\eta_{0}$.
The evolution of the system is determined by the evolution rules and the initial configuration, and we denote this pair by $\xi = \scriptstyle \left(  (X^{x,i}_t)_{t\geq0,i\in\N,x\in\Z^d}, (\eta_0(x))_{x\in\Z^d} \right)$.

Suppose that the system fixates.
Then, necessarily, there exists $k\in \N$ such that $\Pb[\text{the number of particles which ever visit } \o \text{ equals }k]>0$.
Moreover, there exist $x_1,\dots,x_k\in\Z^d$ such that $\Pb[\mathfrak{A}]>0$, where
$$\mathfrak{A} = \left[ \text{the particles which ever visit } \o \text{ are initially at the sites }x_1,\dots,x_k \right].$$

Consider two copies~$\xi$ and~$\tilde{\xi}$ of the system, coupled as follows.
We sample the same evolution rules for~$\xi$ and~$\tilde{\xi}$, and also the same initial configuration outside $\{x_1,\dots,x_k\}$.
The initial configuration in $\{x_1,\dots,x_k\}$ is sampled independently for~$\xi$ and~$\tilde{\xi}$.
Now notice that by independence
\begin{multline*}
\Pb\left[ \mathfrak{A} \mbox{ occurs for } \tilde{\xi} \mbox{ and } \eta_0(x_1)=\cdots=\eta_0(x_k)=0 \mbox{ for } \xi \right]
=
\\
=
\Pb\left[ \mathfrak{A} \mbox{ occurs for } \tilde{\xi} \right] \times \Pb\left[ \big. \eta_0(x_1)=\cdots=\eta_0(x_k)=0 \mbox{ for } \xi \right]
>0
.
\end{multline*}
We conclude the proof with the observation that, on the above event, no particle ever visits~$\o$ in the system~$\xi$.
Indeed, on the above event, the initial configuration of~$\xi$ is the same as that of~$\tilde{\xi}$ except for the deletion of the particles present in $\{x_1,\dots,x_k\}$.
In particular, all the particles which visit the origin in~$\tilde{\xi}$ are deleted in~$\xi$.
Recalling that~$\xi$ and~$\tilde{\xi}$ share the same evolution rules, we leave to the reader to check that in this case no particles can visit~$\o$ in the system~$\xi$.
\end{proof}

\begin{proposition}
\label{prop:particlesfixate}
If the particle-hole model fixates, then every particle eventually settles.
\end{proposition}

\begin{proof}
A more general version of the proposition is the main result in~\cite{amir-gurelgurevich-10}.
Below we give a simpler argument, following ideas from two-type annihilating random walks~\cite{cabezas-rolla-sidoravicius-13b}.

As in the previous proof, we construct the system using the particle-wise randomness, and denote by~$\xi$ the pair of initial configuration and evolution rules from which the process is constructed.
The law of this evolution is invariant under permutation of labels of particles initially present at the same site.
Thus, it suffices to show that, almost surely on the event that $\eta_0(\o) \geq 1$, the first particle born at the origin eventually settles.

Consider two copies $\xi$ and $\tilde{\xi}$ of the system, coupled as follows.
First, use the same initial configuration~$\eta_0$ for $\xi$ and $\tilde{\xi}$.
As for the evolution rules, use the same~$(X^{x,i}_t)_{t\geq 0}$ for~$\xi$ and~$\tilde{\xi}$, except at $(x,i)=(\o,1)$.
Finally, sample $(X^{\o,1})_{t\geq 0}$ and $(\tilde{X}^{\o,1})_{t\geq 0}$ independently, and assign them to $\xi$ and $\tilde{\xi}$, respectively.

Define $\mathcal{B}$ be as a random subset of~$\Z^d$ given by the set of sites which are never visited by a particle in the system $\tilde{\xi}$.
Since $\mathcal{B}$ is a translation-covariant function of $\tilde{\xi}$, which in turn is distributed as a product measure, it follows that $\mathcal{B}$ is ergodic with respect to translations.
Assuming that the system fixates, by Lemma~\ref{lem:originnevervisited} the set $\mathcal{B}$ is a.s.\ non-empty, and moreover it has positive density.

On the event $[\eta_0(\o)\geq 1]$, system~$\xi$ can be obtained from system~$\tilde{\xi}$ by deleting a particle with trajectory $(\tilde{X}^{\o,1})_{t\geq 0}$, and adding another one with trajectory $(X^{\o,1})_{t\geq 0}$.
The effects of deleting a particle may only be propagated as follows.
Label the deleted particle~$\rho_1$.
Since it is now is absent, it will not settle where it would, say at $x_1$ at $t_1$ (if $\rho_1$ would not settle, its deletion has no effect on the other particles).
This may cause another particle~$\rho_2$ to visit $x_1$ after time~$t_1$, and now~$\rho_2$ will settle at $x_1$, whereas without deletion it would have settled at~$x_2$ at~$t_2>t_1$, and so on.
This deletion thus cannot cause sites in~$\mathcal{B}$ to be visited.
Now notice that~$(X^{\o,1}_t)_{t\geq0}$ is independent of~$\mathcal{B}$.
By Lemma~\ref{lemma:hitsaset} below, this trajectory a.s.~hits~$\mathcal{B}$ at some random time~$T$ and random site~$z$.
Therefore, on the system~$\xi$, particle~$(\o,1)$ either settles before time~$T$ or it settles at~$z$ at time~$T$.
\end{proof}

\begin{lemma}
\label{lemma:hitsaset}
Let~$\mathcal{B}$ is a random subset of~$\Z^d$, ergodic for translations in each direction.
Let~$(X_n)_{n=0,1,2,\dots}$ be a random walk on~$\Z^d$ starting at $X_0=\o$, and independent of~$\mathcal{B}$.
Then $\Pb[X_n \in \B \mbox{ i.o.}]=1$.
\end{lemma}
\begin{proof}
Assume for simplicity that $q:=p(\boldsymbol{e}_1)>0$.
For each~$n\in\N_0$, let $d_n=\inf\{j \in \N_0:X_n+j \cdot \boldsymbol{e}_1 \in \mathcal{B}\}$, that is, $d_n$ is the distance from $X_n$ to the first site in~$\mathcal{B}$ lying on the same horizontal line as $X_n$ and to the right of~$X_n$.
Since~$\B$ is ergodic with respect to translations by~$\boldsymbol{e}_1$, we have $\Pb[d_0 < \infty]=1$.
Now notice that $(d_n)_{n\in\N_0}$ is a stationary sequence, and therefore $\Pb[d_n \to \infty]=0$.
Finally, each time $d_n \leq k$, with probability at least~$q^k$ the walk~$X_n$ hits~$\mathcal{B}$ within the next~$k$ steps,
and since the former event must happen infinitely often for some~$k$, so must the latter.
\end{proof}

\begin{proof}
[Proof of Proposition~\ref{prop:abrecurrence}]
Assume that the particle-hole model fixates.
By Lemma~\ref{lem:originnevervisited}, the density of unfilled holes does not decrease to~$0$.
By Proposition~\ref{prop:particlesfixate}, the density of unsettled particles tends to~$0$ as~$t\to\infty$.
Since the system locally preserves the difference between unsettled particles and unfilled holes, the density of unsettled particles minus the density of unfilled holes is constant in time (see the proof of Theorem~\ref{thm:fixationbelowone}).
Hence, the density of unsettled particles at $t=0^-$ is strictly smaller than the density of holes at $t=0^-$, which equals~$1$, proving the proposition.
\end{proof}

\section{Critical flow in one dimension}
\label{sec:criticalflow}

In this section we consider the \emph{flow process}, i.e., the process which counts the amount of particles which have passed through $\boldsymbol o$.
We find the scaling limit of this process for the biased particle-hole model in $\Z$, which is given by the running maximum of a Brownian motion.

A similar scaling limit should hold for the ARW with asymmetric walks at $\lambda=\infty$.
It would be interesting to understand the scaling limit of totally-asymmetric walks with finite~$\lambda$ at critical density $\mu_c=\frac{\lambda}{1+\lambda}$, but we have not been able to find the correct description.
The case of asymmetric walks and finite $\lambda$ is much less clear, let alone that of symmetric walks.

Consider the particle-hole model with jump probabilities $p>\frac{1}{2}$ to the right and $q=1-p$ to the left, and initial condition having mean~$\mu=1$ and positive finite variance~$\sigma^2$.
We define the flow process as
\begin{align*}
C_t:=\text{number of particles which have passed through } \boldsymbol{o} \text{ before time } t, \quad t \geq 0.
\end{align*}
Let $(B_t)_{t\geq0}$ be a one-dimensional Brownian motion started at $0$ and $\tilde{B}_t:=\max\{B_s:s \leq t\}$ denote its running maximum.
The theorem below states that the scaling limit of the flow process $(C_t)_{t\geq0}$ is $(\tilde{B}_t)_{t\geq0}$.
The \textsl{plateaux} of $\tilde{B}$ (given by excursions of $B$ below its running maximum) correspond to the ever longer intervals of inactivity at the origin in the model.
Moreover, the scale invariance $\tilde{B}_{L^2 t} \stackrel{\dd}{=} L \, \tilde{B}_{t}$ indicates that the amount of particles which pass through the origin before time~$t$ is of order $\sqrt{t}$, providing a critical exponent.
The above observations are in agreement with the predictions of vanishing activity and non-fixation.

\begin{theorem}
\label{thm:scalinglimit}
For $d=1$, let $v=p-q>0$ denote the average speed of a moving particle in the particle-hole model.
Assume $\E[\eta_0(\boldsymbol{o})]=1$ and $\E[\eta_0(\boldsymbol{o})^2]=1+\sigma^2$ with $0<\sigma<\infty$.
Then the scaling limit of the flow process~$(C_t)_{t\geq 0}$ is given by
\[
\left(\tfrac{1}{\sigma L} \, C_{ \frac{L^2t}{v}}\right)_{t\geq0}
\stackrel{\dd}{\longrightarrow}
\left( \tilde{B}_t \right)_{t\geq0},
\]
where $\stackrel{\dd}{\to}$ denotes convergence in distribution in $D[0,\infty)$ with the $M_1$-topology.
\end{theorem}

Before presenting the proof of the theorem, we will give a intuitive explanation of the result.
Define
\begin{align*}
 S_t:=\sum_{i=\lfloor-t\rfloor}^0(\eta_{0}(i)-1).
\end{align*}
Observe that $S_n$ is the number of particles minus the number of holes in $[-n,0]$. Assume for simplicity that the system is totally asymmetric and that particles jump at discrete times. Let $n_1:=\min\{n\in\N:S_n>0\}$. Then, all the particles to the right of $-n_1$ will settle before crossing the origin.
Moreover, all the holes to the right of $-n_1$ will be filled by one of such particles. This will create a ``carpet'' that will allow the particles initially at $-n_1$ to achieve the origin. Analogously, setting $n_2:=\min\{k>n_1:S_k>S_{n_1}\}$, all the particles initially at sites $x\in[-n_2+1,-n_1-1]$ will create a carpet in $[-n_1,-n_2]$ over which the particles initially at $-n_2$ will reach $-n_1$ and, therefore, achieve the origin afterwards.
An iteration of this argument gives that, for any $n\in\N$, the amount of particles initially in $[-n,0]$ which ever reach the origin is $\max\{S_k:k\leq n\}$. Assuming that particles travel at speed $v$, we get that $C_t:=\max\{S_k:k\leq vt\}$.

Finally, observe that, under the assumptions of the theorem, $S$ is a random walk whose jump distribution has mean $0$ and finite second moment.
Therefore $S$ scales to a Brownian motion and, consequently, the flux $C$ scales to the maximum of a Brownian motion.
The proof of the theorem consists in making this argument rigorous and valid in the continuous-time, asymmetric setting.

In the remainder of this section we prove Theorem~\ref{thm:scalinglimit} and conclude with a few observations about this scaling limit.

The first step in the proof is to replace the convergence in distribution of rescaled $S$ to Brownian motion by almost sure convergence.
We do this in order to maintain the proof as simple as possible.
Applying Donsker's invariance principle we have that $(\sigma^{-1}\epsilon^{1/2} S_{\epsilon^{-1} t})_{t\geq0}$ converges in distribution to a Brownian motion.
Hence, using Skorohod's representation theorem we have that there exists a coupled sequence of initial configurations, $(\eta^{\epsilon}_0(z))_{z\in\Z}$, and a Brownian motion $(B_t)_{t\geq0}$ defined on a common probability space such that, for all $\epsilon\geq0$, $(\eta^{\epsilon}_0(z))_{z\in\Z}$ is distributed as $(\eta_0(z))_{z\in\Z}$ and
\begin{equation}\label{eq:convergenceofprofiles}
(\sigma^{-1}\epsilon^{1/2}S^{\epsilon}_{\epsilon^{-1} t})_{t\geq0}\stackrel{u}{\to}(B_t)_{t\geq0} \quad \Pb\text{-a.s,}
\end{equation}
 as $\epsilon\to0$, where $\stackrel{u}{\to}$ denotes uniform convergence over compacts and
\begin{align*}
 S^{\epsilon}_t:=\sum_{i=\lfloor-t\rfloor}^0(\eta^{\epsilon}_{0}(i)-1).
\end{align*}
For each $\epsilon>0$, let $(\eta^\epsilon_t)_{t\geq 0}$ be a particle-hole model with initial configuration $(\eta^\epsilon_0(z))_{z\in\Z}$.
We define $C^{\epsilon}_t$ as the amount of particles which pass through $\boldsymbol{o}$ up to time $t$ in the system~$\eta^\epsilon$.

Having constructed the coupling, now we turn our attention to prove that $C^{\epsilon}_{t}$ is close to $\max\{S^{\epsilon}_s:s\leq vt\}$.
More precisely,  we will get a lower bound $C^{\epsilon}_{\epsilon^{-1}t}\geq \max\{S^{\epsilon}_s:s\leq v\epsilon^{-1}t\}-E_1^\epsilon$ and an upper bound $C^{\epsilon}_{\epsilon^{-1}t}\leq \max\{S^{\epsilon}_s:s\leq v\epsilon^{-1}t\}+E_2^\epsilon$, where $E_1^{\epsilon}$ and $E^\epsilon_2$ are negligible terms.

First we will deal with the lower bound.
Let $t^{\ast}$ be the point where $B$ attains his maximum in $[0,vt]$ and
\begin{equation}\label{def:tc}
t^\star:=\min\{s\geq0:B_{s}\geq B_{t^\ast}\}.
\end{equation}
Note that $t^\ast< vt<t^\star$ almost surely.
 By the continuity of the Brownian paths, display~\eqref{eq:convergenceofprofiles}, and the fact that $t^\star>vt$, it follows that the maximum of $S^{\epsilon}$ in the interval $[0,v\epsilon^{-1}t]$ is attained at a point $j^\epsilon\in\N_0$ such that $\epsilon j^{\epsilon}\stackrel{\epsilon\to 0}{\to}t^{\ast}$.

For each particle initially in $[-j^{\epsilon},0]$, one and only one of the three following possibilities must occur:
\begin{enumerate}
\item They pass through the origin,
\item They fill an empty site in $[-j^{\epsilon},0]$ (and stay there forever),
\item They fill an empty site in $(-\infty, -j^{\epsilon}-1]$ (and stay there forever).
\end{enumerate}
Let $E_1^\epsilon$ be the amount of particles for which item $3$ holds.
Since at most one particle can settle at a given site, we have that the particles in item $2$ are at most $j^{\epsilon}$ (which is the number of sites in $[-j^{\epsilon},0]$).
On the other hand, $S^{\epsilon}_{j^{\epsilon}}$ measures the initial difference between particles and sites in $[-j^{\epsilon},0]$.
Hence we have that
\begin{equation}\label{eq:prelowerbound}
\#\{\text{particles initially in }[-j^{\epsilon},0]\text{ which pass through the origin}\}\geq S^{\epsilon}_{j^{\epsilon}}-E_1^\epsilon.
\end{equation}
Let $\B_\epsilon$ be the event that all particles initially in $[-j^{\epsilon},0]$ which pass through the origin, do it before time~$\epsilon^{-1}t$.
Using the fact that the particles perform biased random walks with asymptotic speed $v$ and $\epsilon j^{\epsilon}\stackrel{\epsilon\to0}{\to}t^{\ast}<vt$ it follows that
\begin{equation}\label{eq:neglectB}
 \Pb[\B_\epsilon^c]\stackrel{\epsilon\to0}{\to}0.
\end{equation}
On the other hand, by display \eqref{eq:prelowerbound}, on the event $\B_{\epsilon}$ we have that
\begin{align}\label{eq:lowerbound}
C^{\epsilon}_{\epsilon^{-1}t}\geq S^{\epsilon}_{j^{\epsilon}}-E_1^\epsilon,
\end{align}
which is the desired lower bound.
The next lemma shows that $E_1^\epsilon$ is negligible.
\begin{lemma}\label{lem:neglectE}
 For all $\alpha>\log(p/q)^{-1}$, we have that
$ \Pb[E_1^\epsilon\geq\alpha\log(\epsilon^{-1})]\stackrel{\epsilon\to0}{\to}0.$
\end{lemma}
\begin{proof}
Let $N^\epsilon$ be the number particles initially in $[-j^\epsilon,0]$ (in the system $\eta^\epsilon$).
Let $(Y^i_t)_{t\geq0},i=1,\dots,N^{\epsilon}$ be the trajectories of those particles.
Note that, if one of those particles settles at a site $x<-j^\epsilon$, then necessarily $x\geq\min\{Y^i_t:t\geq0,i=1,\dots,N^{\epsilon}\}$.
Hence, since at most one particle can settle at a given site, we have that $E_1^\epsilon \leq -j^\epsilon-\min\{Y^i_t:t\geq0,i=1,\dots,N^{\epsilon}\}$.
Moreover, since $Y^i_0\geq-j^\epsilon$ for all $i=1,\dots,N^\epsilon$, we have that
\begin{align}\label{eq:keyestimateforEepsilon}
E_1^\epsilon\leq -\min\{Y^i_t-Y^i_0:t\geq0,i=1,\dots,N^{\epsilon}\}.
\end{align}
Furthermore, we have that, for all $i=1,\dots,N^\epsilon$, $(Y^i_t-Y^i_0)_{t\geq0}$ is a biased random walk started at $\boldsymbol o$ (at least up to the time of settlement).
Hence, for the proof of the lemma, first we will control the quantities $N^\epsilon$ and $\min\{Y_t:t\geq0\}$, where $Y$ is a biased random walk started at $\boldsymbol o$ whose jump probabilities are $p$ to the right and $q$ to the left.
We start with the control of $N^{\epsilon}$.
Let $c>1$ be fixed.
Since $\E[\eta_0(\boldsymbol o)]=1$, by the law of large numbers we have that
\begin{align}\label{eq:controlofNepsilon}
\Pb\left[N^\epsilon\geq cj^\epsilon \right]\stackrel{\epsilon\to0}{\to}0.
\end{align}
Now we proceed to control $\min\{Y_t:t\geq0\}$.
Since $Y$ is biased to the right, it follows that $\Pb[\exists s \geq 0 :Y_s=-1]=q/p<1$.
Let $\theta_{z}:=\min\{s\geq0:Y_s=z\}$.
Note that, by repeatedly applying the strong Markov property of $Y$ at the stopping times~$\theta_i$, $i=-1,\dots,-k+1$, we find that
\begin{align}\label{eq:backtracking}
\Pb[\min\{Y_s:s\geq0\}\leq -k]=\left(\frac{q}{p}\right)^k.
\end{align}
Now we are ready to prove the lemma.
By displays \eqref{eq:keyestimateforEepsilon} and \eqref{eq:backtracking}, we can write
\begin{align*}
\Pb\left[E_1^\epsilon\geq\alpha\log(\epsilon^{-1})\vert N^\epsilon< cj^\epsilon \right]\leq cj^\epsilon \left(\frac{q}{p}\right)^{\alpha\log(\epsilon^{-1})},
\end{align*}
which goes to $0$ as $\epsilon\to0$ due to our choice of $\alpha$ and because $j^\epsilon=O(\epsilon^{-1})$.
That, plus \eqref{eq:controlofNepsilon}, proves the lemma.
\end{proof}

We have obtained the desired lower bound for $C^{\epsilon}_{\epsilon^{-1}t}$, now we aim for the corresponding upper bound.
The strategy is to first obtain the upper bound for a truncated version $\eta^{c,\epsilon}$ of our system $\eta^\epsilon$.
Then we will show that the difference of the flow processes of~$\eta^{c,\epsilon}$ and~$\eta^\epsilon$ up to time~$\epsilon^{-1}t$ is negligible.

In order to understand the coupling between~$\eta^\epsilon$ and~$\eta^{c,\epsilon}$, we will assume that the system~$\eta^{\epsilon}$ is given by its initial configuration~$\eta^\epsilon_0$ and a set of \textit{evolution rules}, as in the construction of the particle-hole models in Lemma~\ref{lem:originnevervisited}.
When truncating the system, we will modify only the initial configuration,  the evolution rules will be preserved.

Next, we construct the truncated system.
Define $t^c:=t+\frac{t^\star-t}{2}$, where $t^\star$ is as in \eqref{def:tc}.
For each $\epsilon>0$ let $\eta^{c,\epsilon}$ be the system with initial configuration
$$
\eta^{c,\epsilon}_0(z):=\begin{cases}
                        \eta^{\epsilon}_0(z) &:  z\in [-\lfloor v\epsilon^{-1}t^{c}\rfloor ,0]\\
                         0                   &:  z \notin [-\lfloor v\epsilon^{-1}t^{c}\rfloor ,0],
                     \end{cases}
$$
and the same evolution rules as $\eta^{\epsilon}$.
Define also
$$
S^{c,\epsilon}_t:=\sum_{i=\lfloor- t \rfloor}^0 \eta_0^{c,\epsilon}(i).
$$
 Let the system $\eta^{c,\epsilon}$ evolve until every particle has occupied an empty site (that time exist and is finite because the system $\eta^{c,\epsilon}$ has a finite number of particles). Let $-j^{\ast,\epsilon}$ be the rightmost site in $(-\infty,0]$ which remained empty after the evolution. By definition, all the sites in $[-j^{\ast,\epsilon}+1,0]$ eventually were occupied by a particle. Moreover, those particles must have been initially in the interval $[-j^{\ast,\epsilon}+1,0]$, because no particle ever passed through $-j^{\ast,\epsilon}$ (otherwise, the site would have not remained empty). Hence, since $S^{c,\epsilon}_{j^{\ast,\epsilon}-1}$ measures the initial difference between particles and sites in $[-j^{\ast,\epsilon}+1,0]$, we have that $S^{c,\epsilon}_{j^{\ast,\epsilon}-1}$ gives an upper bound for the amount of particles initially in $[-j^{\ast,\epsilon}+1,0]$ which passed through the origin. Furthermore, since all particles in $[-j^{\ast,\epsilon}+1,0]$ which do not settle in $[-j^{\ast,\epsilon}+1,0]$ must pass through $\boldsymbol o$, we have that $S^{c,\epsilon}$ is, in fact, equal to the number particles initially in $[-j^{\ast,\epsilon}+1,0]$ which passed through the origin. On the other hand, since no particle in $(-\infty,-j^{\ast,\epsilon}]$ ever crossed the origin (because no particle ever pass through $-j^{\ast,\epsilon}$), we have that
\begin{align*}
\#\{\text{particles that pass through }\boldsymbol o \text{ in the system }\eta^{c,\epsilon}\}= S^{c,\epsilon}_{j^{\ast,\epsilon}+1}.
\end{align*}
Since $t^c<t^\star$, we have that the maximum of $S^{c,\epsilon}$ in $\N_0$ (i.e., the global maximum) is attained at $j^\epsilon$, for $\epsilon$ small enough (recall that $j^\epsilon$ is the point at which $S^\epsilon$ attains his maximum in $[0,v\epsilon^{-1}t]$). Hence, by the display above we get that
\begin{align*}
\#\{\text{particles that pass through }\boldsymbol o \text{ in the system }\eta^{c,\epsilon}\}\leq S^{\epsilon}_{j^\epsilon},
\end{align*}
for $\epsilon$ small enough,
which clearly implies that
\begin{align}\label{eq:upperbound}
C^{c,\epsilon}_t\leq S^{\epsilon}_{j^\epsilon},
\end{align}
for $\epsilon$ small enough, where
\begin{align*}
 C^{c,\epsilon}_t:=\#\{\text{particles that pass through }\boldsymbol o \text{ in the system }\eta^{c,\epsilon} \text{ before time }t\}.
\end{align*}

Having obtained the upper bound for the truncated system, we turn our attention to control the difference between the flow processes of the truncated and original systems.
First, we will explain how the differences between the systems $\eta^\epsilon$ and $\eta^{c,\epsilon}$ evolve according to a set of \textit{tracers}.
As a warm up, first we will explain how evolve the difference between systems which differ by a single particle.
Let $\eta^1$ and $\eta^2$ be particle-hole models which evolve under the same evolution rules and whose initial configurations differ by a single particle $a^1$, that is, there exists $x\in\Z$ such that $\eta^1_0(x)=\eta^2_0(x)+1$ and $\eta^1_0(y)=\eta^2_0(y)$ for all $x\neq y$. We will define a \textit{tracer} $(Y^{x}_t)_{t\geq0}$ which will follow the difference due to $a^1$ (the extra particle at $\eta^1$). We set $Y^x_0=x$ and, initially, $Y^x$ will follow the trajectory of $a^1$ until it settles at an empty site $z$. Note that on the system $\eta^2$ the site $z$ remains empty.  Eventually, a particle $a^2$ will settle at $z$ in the system $\eta^2$. However, $a^2$ will not settle at $z$ in the system $\eta^1$, because $z$ was already occupied by $a^1$. At that time, our tracer $Y^{x}$ will start to follow the path of $a^2$ (in the system $\eta^1$). The tracer continues to follow the path of $a^2$ until it settles. We can indefinitely continue this procedure to obtain a tracer $Y^x$ defined for all times with the property that, for all $t\geq0$, we have that $\eta^1_t(Y^x_t)=\eta^2_t(Y^x_t)+1$ and $\eta^1_t(y)=\eta^2_t(y)$ for all $y\neq Y^x_t$.  Moreover, the tracer perform a continuous time random walk with the same transition probabilities as the particles, with the only difference that the tracer is ``stopped" when it is tagging a settled particle.

The initial difference between $\eta^\epsilon$ and $\eta^{c,\epsilon}$ consists in an infinite amount of particles present in $\eta^\epsilon$ and absent at $\eta^{c,\epsilon}$. Using the same procedure as above, we can simultaneously define an infinite family of tracers (one for each particle present at $\eta^\epsilon$ and absent at $\eta^{c,\epsilon}$) which give the evolution of the differences between the systems.
Let $N_t$ be the number of times that one of those tracers pass trough $\boldsymbol o$ up to $t$. Hence, since the tracers give the evolution of the difference between $\eta^\epsilon$ and $\eta^{c,\epsilon}$, we have that
\begin{align}\label{eq:difference}
C^{\epsilon}_t-C^{c,\epsilon}_t\leq  N_t.
\end{align}

Those tracers can be of two types
\begin{enumerate}
\item Starting in $(-\infty,-\lfloor v\epsilon^{-1}t^c\rfloor-1]$,
\item Starting in $[1,\infty)$.
\end{enumerate}
Let $E_2^\epsilon$ be the number of times that a tracer starting at $[1,\infty)$ pass through $\boldsymbol o$.
Let
$$
\mathcal{D}_\epsilon:=\{\text{No tracer starting in }[-\infty,-\lfloor v\epsilon^{-1}t^c \rfloor-1]\text{ reaches } \boldsymbol o \text{ before time } \epsilon^{-1}t\}.
$$
Since $t^c>vt$ and the (unsettled) particles perform biased random walks with asymptotic speed $v$, we have that
\begin{align}\label{eq:neglectD}
\Pb[\mathcal{D}_\epsilon^c]\stackrel{\epsilon\to0}{\to}0.
\end{align}
Moreover, by displays \eqref{eq:upperbound} and \eqref{eq:difference}, on the event $\mathcal{D}_\epsilon$ we have that
\begin{align}\label{eq:upperbound2}
C^{\epsilon}_{\epsilon^{-1}t}\leq S^{\epsilon}_{j^\epsilon}+ E_2^\epsilon.
\end{align}
The following lemma shows that $E_2^\epsilon$ is negligible.
\begin{lemma}\label{lem:neglectE2}
$ \E[E_2^\epsilon]\leq v^{-1}\frac{q}{2p}$
\end{lemma}
\begin{proof}
Since the tracers are either tagging a particle or an empty site, their trajectories are time changes of biased random walks. Hence, using display \eqref{eq:backtracking}, we find that
\begin{align*}
\E[\#\{\text{tracers starting in }[1,\infty)\text{ which visit }\boldsymbol{o}\}]\leq\sum_{i=1}^\infty\E[\eta_0(i)]\left(\frac{q}{p}\right)^i=\frac{q}{2p}.
\end{align*}
On the other hand, using the same tools, we get that the expected number of visits to $\boldsymbol o$ of each one of the tracers which visit $\boldsymbol o$ equals $v^{-1}$
Hence we have that
\begin{equation*}
\E[E_2^{\epsilon}]\leq v^{-1}\frac{q}{2p}.
\qedhere
\end{equation*}
\end{proof}
We have obtained the desired lower and upper bounds. We are ready to prove the following lemma.
\begin{lemma}
\label{lem:Pconvergence}
For all $t\geq0$, we have that
$$\epsilon^{1/2}C^{\epsilon}_{\epsilon^{-1}t}\stackrel{P}{\to}\max\{\sigma B_s:s\leq vt\}\quad\text{as }\epsilon\to0,$$
where $\stackrel{P}{\to}$ denotes convergence in probability and $\sigma$ is as in Theorem~\ref{thm:scalinglimit}.
\end{lemma}
\begin{proof}
   First note that, by displays \eqref{eq:lowerbound} and \eqref{eq:upperbound2}, on event $\B_\epsilon\cap\mathcal{D}_\epsilon$ we have that
  \begin{align*}
 S^\epsilon_{j^\epsilon}-E_1^\epsilon\leq C^{\epsilon}_{\epsilon^{-1}t}\leq S^{\epsilon}_{j^{\epsilon}} + E_2^\epsilon.
  \end{align*}
  Hence, using Lemmas~\ref{lem:neglectE} and~\ref{lem:neglectE2} we get that, for any $\delta>0$
  \begin{align*}
  \Pb[|\epsilon^{1/2}C^\epsilon_{\epsilon^{-1}t}-\epsilon^{1/2}S^{\epsilon}_{j^{\epsilon}}|\geq\delta\big\vert\B_\epsilon\cap\mathcal{D}_\epsilon]\stackrel{\epsilon\to0}{\to}0.
  \end{align*}
  By the display above and displays \eqref{eq:neglectB} and \eqref{eq:neglectD}, we get that
  \begin{align*}
  \Pb[|\epsilon^{1/2}C^\epsilon_{\epsilon^{-1}t}-\epsilon^{1/2}S^{\epsilon}_{j^{\epsilon}}|\geq\delta]\stackrel{\epsilon\to0}{\to}0.
  \end{align*}
  Furthermore, recalling the fact that $\epsilon j^\epsilon\stackrel{\epsilon\to0}{\to}t^\ast$ and display \eqref{eq:convergenceofprofiles}, we get that
  \begin{align*}
  \Pb[|\epsilon^{1/2}C^\epsilon_{\epsilon^{-1}t}-\sigma B_{t^\ast}|\geq\delta]\stackrel{\epsilon\to0}{\to}0.
  \end{align*}
  That is, $\epsilon^{1/2}C^\epsilon_{\epsilon^{-1}t}$ converges in probability to $B_{t^\ast}=\max\{B_s:s\leq vt\}$.
  \end{proof}
Using the previous lemma we now show Theorem~\ref{thm:scalinglimit}.Let $l\in\N$ and $0\leq t_1\leq t_2 \leq \dots \leq t_l\leq\infty$.
Applying Lemma~\ref{lem:Pconvergence} at times $v^{-1}t_i,i=1,\dots,l$ we get that
\begin{align}\label{eq:finitedimensional}
(\sigma^{-1}\epsilon^{1/2}C^{\epsilon}_{v^{-1}\epsilon^{-1} t_1}\!,\dots,\!\sigma^{-1}\epsilon^{1/2}C^{\epsilon}_{v^{-1}\epsilon^{-1} t_l})
\stackrel{P}{\to}(\max\{\!B_s:s\leq t_1\!\}\!,\dots,\!\max\{\!B_s:s\leq t_l\!\}),
\end{align}
as $\epsilon\to0$.
Since convergence in probability implies convergence in distribution and, for each $\epsilon>0$, $(C^{\epsilon}_t)_{t\geq0}$ is distributed as  $(C_t)_{t\geq0}$, we have that display \eqref{eq:finitedimensional} implies the convergence of the finite-dimensional distributions of $(\sigma^{-1}\epsilon^{1/2}C_{v^{-1}\epsilon^{-1} t})_{t\geq0}$ to those of $(\max\{B_s:s\leq t\})_{t\geq0}$.
On the other hand, the function $t \mapsto C_t$ is monotone, hence convergence of finite-dimensional  distributions implies convergence in $(D[0,\infty),M_1)$ (see \cite[Theorem 12.12.3]{whitt-02}).
This finishes de proof of the theorem.

The above theorem provides a rather complete description of the large-scale behavior of the model at the critical density.
Recalling that $\tilde{B}$ can also be expressed as the inverse of an $\alpha$-stable subordinator, with $\alpha=\frac{1}{2}$, we see that our result also provides scaling exponents.

We expect the same scaling limit for the flow of the asymmetric ARW with $\lambda=\infty$ and $\mu=1$.
Nevertheless, for the asymmetric ARW with $\lambda<\infty$ at criticality, we expect an intermittency between periods of inactivity followed by bursts of high activity.
That is, the system should display an ``avalanche"-type of relaxation after periods of load of particles.
This should be reflected in a discontinuous scaling limit of the flow process, with discontinuities corresponding to avalanches.
Since $\tilde{B}$ is continuous, we expect a different scaling limit in that case.

\section*{Acknowledgments}

We thank C.~Hoffman, R.~Dickman, and G.~Kozma for inspiring discussions.
The research reported in this paper was supported in part by US NSF grant DMS-1007626, IMPA, CONICET, ANPCyT, and Brazilian grant CNPq-PDJ 150897/2012-0.

\bibliographystyle{bib/rollaalphasiam}
\addcontentsline{toc}{section}{References}
\bibliography{bib/leo}

\end{document}

%% file: macros.tex
\usepackage{amscd}
\usepackage{amsfonts}
\usepackage{amsmath}
\usepackage{amssymb}
\usepackage{amstext}
\usepackage{amsthm}
\usepackage{color}
\usepackage[nodate]{datetime}
\usepackage{dsfont}
\usepackage{hyperref}
\usepackage[latin1]{inputenc}
\usepackage{mathtools}
\usepackage{srcltx}
\usepackage[normalem]{ulem}
\usepackage{verbatim}

\usepackage{fontenc}
\usepackage{lmodern}

\usepackage[lmargin=25mm,rmargin=25mm,top=25mm,bottom=30mm,paperwidth=210mm,paperheight=11in]{geometry}

\newcommand{\mydate}{{\today}}

\date{\mydate}

\newtheorem{theorem}{Theorem}
\newtheorem{proposition}{Proposition}
\newtheorem{lemma}{Lemma}

\newtheorem{remark*}{Remark}

\newcommand{\E}{\mathbb{E}}
\newcommand{\N}{\mathbb{N}}
\newcommand{\Pb}{\mathbb{P}}

\newcommand{\Z}{\mathbb{Z}}

\newcommand{\B}{\mathcal{B}}

\newcommand{\dd}{{\mathrm d}}

\renewcommand{\leq}{\leqslant}

\renewcommand{\geq}{\geqslant}

\renewcommand{\o}{\boldsymbol{o}}
\newcommand{\I}{\mathds{1}}

%% file: arwcritical.bbl
\newcommand{\etalchar}[1]{$^{#1}$}
\begin{thebibliography}{BBB{\etalchar{+}}12}

\bibitem[AGG10]{amir-gurelgurevich-10}
{\sc G.~Amir and O.~Gurel-Gurevich}, {\em On fixation of activated random
  walks}, Elect. Comm. in Probab., 15 (2010), pp.~119--123.

\bibitem[BBB{\etalchar{+}}12]{basu-basu-bondyopadhyay-mohanty-hinrichsen-12}
{\sc M.~Basu, U.~Basu, S.~Bondyopadhyay, P.~K. Mohanty, and H.~Hinrichsen},
  {\em Fixed-energy sandpiles belong generically to directed percolation},
  Phys. Rev. Lett., 109 (2012), pp.~44--48.

\bibitem[CRS13]{cabezas-rolla-sidoravicius-13b}
{\sc M.~Cabezas, L.~T. Rolla, and V.~Sidoravicius}, {\em Recurrence and density
  decay for diffusion-limited annihilating systems}.
\newblock Submitted. arXiv:1309.4387, 2013.

\bibitem[DRS10]{dickman-rolla-sidoravicius-10}
{\sc R.~Dickman, L.~T. Rolla, and V.~Sidoravicius}, {\em Activated random
  walkers: Facts, conjectures and challenges}, J. Stat. Phys., 138 (2010),
  pp.~126--142.

\bibitem[Lee13]{lee-13}
{\sc S.~B. Lee}, {\em Comment on ``fixed-energy sandpiles belong generically to
  directed percolation''}, Phys. Rev. Lett., 110 (2013), p.~159601.

\bibitem[MD99]{marro-dickman-99}
{\sc J.~Marro and R.~Dickman}, {\em Nonequilibrium phase transitions in lattice
  models}, Collection Al{\'e}a-Saclay: Monographs and Texts in Statistical
  Physics, Cambridge University Press, Cambridge, 1999.

\bibitem[RS12]{rolla-sidoravicius-12}
{\sc L.~T. Rolla and V.~Sidoravicius}, {\em Absorbing-state phase transition
  for driven-dissipative stochastic dynamics on {$Z$}}, Invent. Math., 188
  (2012), pp.~127--150.
\newblock arXiv:0908.1152.

\bibitem[She10]{shellef-10}
{\sc E.~Shellef}, {\em Nonfixation for activated random walks}, Alea, 7 (2010),
  pp.~137--149.

\bibitem[Whi02]{whitt-02}
{\sc W.~Whitt}, {\em Stochastic-process Limits}, Springer Series in Operations
  Research, Springer-Verlag, New York, 2002.
\newblock An introduction to stochastic-process limits and their application to
  queues.

\end{thebibliography}
